\documentclass[a4paper,12pt,reqno]{amsart}
\usepackage{latexsym,amscd,amssymb,amsmath,amsthm,comment,url}
\usepackage[dvips]{graphicx}
\usepackage[all]{xy}
\xyoption{poly}
\xyoption{knot}
\xyoption{import}
\makeatletter
\theoremstyle{plain}
  \newtheorem{thm}{Theorem}[section]
  \newtheorem{prop}[thm]{Proposition}
  \newtheorem{lem}[thm]{Lemma}
  \newtheorem{cor}[thm]{Corollary}
  
\theoremstyle{definition}
  \newtheorem{dfn}[thm]{Definition}
  \newtheorem{exmp}[thm]{Example}
\theoremstyle{remark}
  \newtheorem{rem}[thm]{Remark}
\let\opn\operatorname 
\let\term\emph
\def\@bothmode#1{\ifmmode #1\else $#1$\fi}
\newcount\@tempchn 
\def\@chCount#1{%
   \@tempchn=0
   \@tfor\member:=#1\do{\advance\@tempchn by 1}%
}
\def\@autopr#1{%
   \@chCount{#1}%
   \ifnum\@tempchn<2 #1\else (#1)\fi
}
\let\@tempopn\relax 
\def\@opform_#1#2{\@tempopn_{#1}\@autopr{#2}}
\numberwithin{equation}{section}
\def\NN{\mathbb{N}} 
\def\ZZ{\mathbb{Z}} 
\def\RR{\mathbb{R}} 
\def\kk{\Bbbk} 
\def\m{\ideal{m}} 
\def\p{\ideal{p}} 
\let\e\varepsilon 
\let\s\sigma 
\let\C\Sigma 
\let\t\tau 
\let\u\upsilon 
\let\i\iota 
\def\MM{\mathcal M} 
\def\M{\mathbf M} 
\let\@tempar\relax 
\def\@seton^#1{\overset{#1}{\@tempar}}
\def\defar#1#2{\@xp\def\csname #1\endcsname{\def\@tempar{#2}\@ifnextchar^{\@seton}{\@tempar}}}
\defar{longto}{\longrightarrow} 
\defar{epito}{\twoheadrightarrow} 
\defar{monoto}{\rightarrowtail} 
\defar{embto}{\hookrightarrow} 
\def\imply{\@bothmode{\Rightarrow}} 
\def\Imply{\@bothmode{\Longrightarrow}} 
\def\iff{\@bothmode\Longleftrightarrow} 
\def\get{\@bothmode{\Leftarrow}} 
\def\Get{\@bothmode{\Longleftarrow}} 
\def\bra#1{[#1]} 
\def\mbra#1{\{ #1\}} 
\def\set#1#2{\mbra{\,#1\mid #2\,}} 
\let\Dsum\bigoplus 
\let\tns\otimes 
\def\sM{|\MM |} 
\def\zM{\ZZ\MM} 
\def\szM{|\ZZ\MM |} 
\def\supp{\opn{supp}} 
\def\cell{\mathcal X} 
\def\op{\mathsf{op}} 
\def\sph{\mathbb S} 
\def\defopn#1{%
    \@xp\def\csname #1\endcsname{%
        \def\@tempopn{\opn{\csname the#1\endcsname}}%
        \@ifnextchar_{\@opform}{\@opform_{}}%
    }%
}

\defopn{Ker} 

\defopn{Im} 

\defopn{Cok} 

\defopn{Ann} 

\defopn{Ass} 

\defopn{Spec} 

\defopn{pd} 

\defopn{id} 
\def\idmap{\opn{id}} 
\def\the@init{in}
\defopn{@init}
\def\init{\@init_{\succ}}
\def\fring#1{\kk \bra{#1}} 
\let\ideal\mathfrak 
\def\theE{E}
\def\E{\@ifstar{{}^*\theE}{\theE}} 
\let\defcat\defopn
\def\theMod{Mod}    \def\themod{mod}
      
\let\the@Lgr\theMod  \let\the@lgr\themod
\defcat{Mod} 
\defcat{mod} 
\defcat{@Lgr}
\def\Lgr#1{\@Lgr_{\ZZ\MM}{#1}} 
\defcat{@lgr}
\def\lgr#1{\@lgr_{\ZZ\MM}{#1}} 
\defcat{Sq} 
\defcat{InjSq} 
\let\colimit\varinjlim 
\def\theHom{Hom}    \def\theRHom{RHom}
\def\theExt{Ext}    
\def\theD{D}
\let\@tempgrop\underline
\def\Hom{\@ifstar{\opn{\@tempgrop\theHom}}{\opn\theHom}} 
\def\RHom{\@ifstar{\opn{R\@tempgrop\theHom}}{\opn\theRHom}} 
\def\Ext{\@ifstar{\opn{\@tempgrop\theExt}}{\opn\theExt}} 
\def\uExt{\@ifstar{\opn{\@tempgrop\theuExt}}{\opn\theExt}} 
\def\uExt{\underline{\operatorname{Ext}}}
\def\inHom{\Hom^\bullet}

\def\theDcat{{\mathsf D}}
\def\Db{\theDcat^b} 
\def\@G_#1{\Gamma_{#1}}
\def\G{\@ifnextchar_{\@G}{\@G_\m}} 
\def\DD{\mathbb D} 

\def\theHtcat{{\mathsf K}}
\def\Cb{\theHtcat^b} 

\def\cpx#1{#1^{\bullet}} 
\def\theD{D} 
\def\D{\@ifstar{{}^*\theD}{\theD}} 
\def\<{{\langle}}
\def\>{{\rangle}}
\def\ts{t_\s}
\def\J{J^\bullet}

\def\DC{D^\bullet}
\def\L{L^\bullet}
\def\bM{\overline{\M}}
\def\bMM{\overline{\MM}}
\def\oR{\widetilde{R}}
\def\oP{\widetilde{\p}}

\def\fs{\fring{\s}}
\def\ofs{\overline{\fs}}
\makeatother
\title{Dualizing complex of a toric face ring II: non-normal case}

\author{Kohji Yanagawa}
\thanks{Partially supported by Grant-in-Aid for Scientific Research (c) (no.19540028).}
\address{Department of Mathematics, Kansai University,
Suita 564-8680, Japan}
\email{yanagawa@ipcku.kansai-u.ac.jp}
\begin{document}
%
%
\maketitle
\begin{abstract}
The notion of {\it toric face rings} generalizes both 
Stanley-Reisner rings and affine  semigroup rings, and has been 
studied by Bruns, R\"omer, et.al.  
Here, we will show that, for a toric face ring $R$, 
the ``graded" Matlis dual of a C\v ech complex gives a dualizing complex. 
In the most general setting, $R$ is not a graded ring in the usual sense.  
Hence technical argument is required.
\end{abstract}

\section{Introduction}
Stanley-Reisner rings and affine semigroup rings are important subjects of 
combinatorial commutative algebra. 
The notion of {\it toric face rings}, which originated in an earlier work of 
Stanley \cite{St}, generalizes both of them, and has been 
studied by Bruns, R\"omer, and their coauthors recently (e.g. \cite{BG, BKR, IR}). 
Contrary to these classical examples,  
a toric face ring does not admit a nice multi-grading in its most general setting. 
To make a toric face ring $R$ from a finite regular cell complex $\cell$, we  
assign each cell $\s \in \cell$ an affine semigroup $\M_\s \subset \ZZ^{\dim \s +1}$ 
with $\ZZ\M_\s = \ZZ^{\dim \s +1}$ so that certain compatibility is satisfied, 
and ``glue" the affine semigroups $\kk[\M_\s]$ of $\M_\s$ along with $\cell$. 
(Note that not all $\cell$ can support toric face rings.)

In the previous paper (\cite{OY}), 
Okazaki and the author gave a concise description of a dualizing complex of 
$R$ under the assumption that $\kk[\M_\s]$ is normal for all $\s \in \cell$.  
In the present paper, we treat the general (i.e., non-normal) case. 
While the result in \cite{OY} does not hold verbatim, we can show that  
the ``Matlis dual"  $(\cpx L_R)^\vee$ of the C\v ech complex $\cpx L_R$
associated with the cell complex $\cell$ is quasi-isomorphic to the dualizing complex. 
If $R$ itself is an affine-semigroup ring, $(\cpx L_R)^\vee$ is the multigraded dualizing 
complex given by \cite{I}. More generally, 
if $R$ has a nice-multigrading, this fact was already proved by Ichim and R\"omer \cite{IR}. 
In their case, standard argument using the graded ring structure works,  but  
the general case requires much more technical argument.  
It would be an interesting problem to find another class of rings whose dualizing 
complexes are given by a similar way.  

\section{Notation and Preliminaries}
In this section, we recall the construction and basic properties of a toric face ring. 
See \cite{BKR,OY} for detail. We basically use the convention of \cite{OY}.    

\medskip

Let $\cell$ be a finite regular cell complex with the 
intersection property, and $X$ its underlying topological space. 
More precisely, $\cell$ is a finite set of subsets (called {\it cells}) of $X$ 
satisfying the following conditions.   
\begin{enumerate}
\item $\emptyset \in \cell $, 
$X = \bigcup_{\sigma \in \cell } \sigma$, and the cells $\sigma \in \cell $ are pairwise disjoint;
\item If $\emptyset \ne \sigma \in \cell $, then, for some $i \in \NN$, 
there exists a homeomorphism from an $i$-dimensional ball 
$\{ x \in \RR^i \mid ||x|| \leq 1 \}$ to the closure 
$\overline{\sigma}$ of $\sigma$ which maps $\{ x \in \RR^i \mid ||x|| < 1 \}$ onto $\sigma$; 
\item For $\sigma \in \cell$, the closure 
$\overline{\sigma}$ is the union of some cells in $\cell$; 
\item For $\s,\t \in \cell$, there is a cell $\u \in \cell$ such that 
$\overline \u = \overline \s \cap \overline \t$ (here $\u$ can be $\emptyset$).
\end{enumerate}  

A simplicial complex is a  typical example of our $\cell$. 
We regard $\cell$ as a partially ordered set ({\it poset} for short) by $\s \geq \t \stackrel{\text{def}}
{\Longleftrightarrow} \overline \s \supset \t$.

\begin{dfn}\label{sec:cell_cpx_ver}
A \term{conical complex} $(\C, \cell)$ on $\cell$ consists of the following data. 
\begin{enumerate}
\item $\C = \{ \, C_\s \mid \s \in \cell \, \}$, where 
$C_{\s} \subset \RR^{\dim \s + 1}$ is a polyhedral cone 
with $\dim C_{\s} = \dim \s + 1$. (In this paper, ``cone" means the one containing no line.)
\item An injection $\i_{\s,\t}:C_{\t} \to C_{\s}$ for $\s, \t \in \cell$ with 
$\s \geq \t$ satisfying the following. 
\begin{enumerate}
\item $\i_{\s,\t}$ can be lifted up to a linear map  
$\tilde{\i}_{\s,\t} :\RR^{\dim \t+1} \to \RR^{\dim \s +1}$. 
\item  The image $\i_{\s, \t}(C_{\t})$ is a face of $C_{\s}$.  
Conversely, for a face $C'$ of $C_\s$, there is a {\it sole} cell $\t$ with $\t \leq \s$
such that $\i_{\s,\t}(C_\t) = C'$. 
\item $\i_{\s, \s} = \idmap_{C_{\s}}$ and $\i_{\s,\t} \circ \i_{\t,\u}$ = $\i_{\s,\u}$ for $\s,\t,\u \in \cell$ with $\s \geq \t \geq \u$.
\end{enumerate}
\end{enumerate}
\end{dfn}

A polyhedral fan in $\RR^n$ gives a conical complex. 
In this case, as an underlying cell complex, we can take 
$\{ \, \text{the relative interior of $C \cap \sph^{n-1}$} \mid C \in \C \, \}$, where $\sph^{n-1}$ 
is the $(n-1)$-dimensional unit sphere in $\RR^n$, and the injections $\i$ are inclusion maps.

\begin{exmp}
Consider the following  cell decomposition of a M\"obius strip. 
\begin{figure}[h]\label{sec:Moebius}
$
\xy /r2.0pc/:,
{\xypolygon3"A"{~={75}~:{(-1,1.7)::}~>{}\bullet}},
+(.8,.8),
{\xypolygon3"B"{~={75}~:{(-1,1.7)::}~>{}\bullet}},
{"A1"\PATH~={**@{-}}'"A2"'"A3"'"B3"'"B2"'"B1"'"A1"},
"A2";"B2"**@{-},
{\vtwist~{"A1"}{"B1"}{"A3"}{"B3"}},
"A1"*+!RD{x}, "A2"*+!R{y}, "A3"*+!LU{z},
"B1"*+!RD{u}, "B2"*+!R{v}, "B3"*+!LU{w}
\endxy
$
\end{figure}
\noindent Regarding each rectangles as the cross-sections of 3-dimensional cones,
we have a conical complex that is not a fan (see \cite{BG}).
\end{exmp}

Let $\RR^{\dim \s+1}$ be the space containing $C_\s$, and 
$\ZZ^{\dim \s +1} \subset \RR^{\dim \s +1}$ the set of lattice points. 
Assume that $\tilde{\i}_{\s, \t}(\ZZ^{\dim \t +1}) \subset \ZZ^{\dim \s +1}$ 
for all $\s, \t \in \cell$ with $\s \geq \t$.

\begin{dfn}
A \term{monoidal complex} $\MM$ supported by a conical complex $(\C,\cell)$ is a set of monoids
$\mbra{\M_\s}_{\s \in \cell}$ with the following conditions:
\begin{enumerate}
\item $\M_\s \subset \ZZ^{\dim \s+1}$ for each $\s\in \cell$, 
and it is a finitely generated additive submonoid (so $\M_\s$ is an affine semigroup) with 
$\ZZ\M_\s = \ZZ^{\dim \s +1}$;
\item $\M_\s \subset C_\s$ and $\RR_{\geq 0} \M_\s = C_\s$ for each $\s\in \cell$; 
\item for $\s,\t \in \cell$ with $\s \ge \t$, the map $\i_{\s,\t}:C_\t \to C_\s$ induces an isomorphism
$\M_\t \cong \M_\s \cap \i_{\s,\t}(C_\t)$ of monoids.
\end{enumerate}
\end{dfn}

For example, let $\C$ be a rational fan in $\RR^n$. Then $\set{C \cap \ZZ^{n}} {C \in \C}$ gives 
a monoidal complex. More generally, taking a suitable submonoid of $C \cap \ZZ^{n}$ for each $C \in \C$, 
we get a monoidal complex whose monoids are not normal.

For a monoidal complex $\MM$, set
$$
\sM := \colimit_{\s \in \cell}\M_\s \quad \text{and} 
\quad \szM := \colimit_{\s \in \cell}\ZZ \M_\s,
$$
where the direct limits are taken with respect to  $\i_{\s, \t} : \M_\t \to \M_\s$ and
$\tilde{\i}_{\s, \t}:\ZZ \M_\t \to \ZZ \M_\s$ for $\s, \t \in \cell$ with $\s \ge \t$.
Note that $\sM$ (resp. $\szM$) is just a set and no longer a monoid (resp. abelian group) in general.  
Since all $\i_{\s,\t}$ (resp. $\tilde{\i}_{\s, \t}$) is injective, we can regard $\M_\s$ (resp. $\ZZ \M_\s$) 
as a subset of $\sM$ (resp. $\szM$). 
For example, if $\MM$ comes from a fan in $\RR^n$, then $\sM = \bigcup_{\s \in \cell} \M_\s \subset \ZZ^n$.

Let $a, b \in \szM$. If there is some $\s \in \cell$ with $a, b \in \ZZ \M_\s$, 
by our assumption on $\cell$, there is a unique minimal cell among 
these $\s$'s. Hence we can define $a \pm b \in \ZZ \M_\s \subset \szM$. 
If there is no $\s \in \cell$ with $a,b \in \ZZ \M_\s$, then $a \pm b$ do not exist.

\begin{dfn}[\cite{BKR}]\label{non-multi-graded }
Let $\MM$ be a monoidal complex, and $\kk$ a field. Then the $\kk$-vector space
$$
\fring \MM := \Dsum_{a \in \sM} \kk \, t^a,
$$
where $t$ is a variable, equipped with the following multiplication
$$
t^a \cdot t^b = \begin{cases}
t^{a+b} & \text{if $a+b$ exists,}\\
0       & \text{otherwise,}
\end{cases}
$$
has a $\kk$-algebra structure. 
We call $\fring \MM$ the \term{toric face ring} of $\MM$ over $\kk$.
\end{dfn}

Clearly,  $\dim \kk[\MM] = \dim \cell +1$. In the rest of this paper, we set $d:= \dim \kk[\MM]$.  
Stanley-Reisner rings and affine semigroup rings (of 
positive semigroups) can be established as toric face rings.
If $\MM$ comes from a fan in $\RR^n$, then $\kk[\MM]$ admits a $\ZZ^n$-grading 
with $\dim_\kk \kk[\MM]_a \leq 1$ for all $a \in \ZZ^n$. 
But this is not true in general.

\begin{exmp}[{\cite[Example 4.6]{BKR}}]\label{sec:Moebius_ring}
Consider the conical complex given in Example \ref{sec:Moebius}. 
Assigning  normal semigroup rings of the form $\kk[a,b,c,d]/(ac-bd)$ 
to the three rectangles, we have a toric face ring of the form 
$$\kk[x, y, z, u, v, w]/(xv - uy, vz - yw, 
xz - uw, uvw, uvz),$$
which does not admit a nice multi-grading. 
We can also get a similar example whose $\kk[\M_\s]$ are not normal. 
\end{exmp}

Let $R := \fring \MM$ be a toric face ring, and $\Mod R$ the category of $R$-modules. 

\begin{dfn}
$M \in \Mod R$ is said to be {\it $\zM$-graded} if the following are satisfied;
  \begin{enumerate}
  \item $M = \Dsum_{a \in \szM}M_a$ as $\kk$-vector spaces;
  \item Let $a \in \sM$ and $b \in \szM$. If $a+b$ exists,  then 
$t^a \, M_b \subset M_{a + b}$.  Otherwise, $t^a \, M_b = 0$.
  \end{enumerate}
\end{dfn}

Since  $R$ may not be a graded ring in the usual sense, the word ``$\zM$-graded" is 
abuse of terminology.
An ideal of $R$ is $\zM$-graded if and only if it is generated by 
monomials (i.e., elements of the form $t^a$).

Let $\Lgr R$ denote the subcategory of $\Mod R$ whose objects are $\zM$-graded
and morphisms are degree preserving (i.e., $f:M \to N$ with 
$f(M_a) \subset N_a$ for all $a \in \szM$). 
It is clear that $\Lgr R$ is an abelian category.

\medskip

For $\s \in \cell$,  a monomial ideal $\p_\s := (t^a \mid a \in \sM \setminus \M_\s)$ of $R$ 
is prime. In fact, the quotient ring $\kk[\s]:=R/\p_\s$ is isomorphic to the affine semigroup ring 
$\kk[\M_\s]$.  Conversely, any monomial prime ideal of $R$ is of the form $\p_\s$ for some $\s \in \cell$. 

We say $R$ is \term{cone-wise normal}, if $\kk[\s]$ is normal (equivalently, 
$\M_\s = C_\s \cap \ZZ^{\dim \s +1}$) 
for all $\s \in \cell$. Set $$I^{-i}_R := 
\Dsum_{\substack{\s \in \cell \\ \dim\kk[\s] = i}} \kk[\s]$$
for $i = 0, \dots , d$, and define $\partial:I^{-i}_R \to I^{-i+1}_R$
by
$$
\partial(x) = \sum_{\substack{\dim \fring{\t} =i-1 \\ \t \leq \s}}\e(\s,\t) \cdot 
f_{\t,\s}(x)
$$
for $x \in \kk[\s] \subset I^{-i}_R$,
where  $f_{\t,\s}$ is the natural surjection $\kk[\s] \to \kk[\t]$
(note that if $\t \leq \s$ then $\p_\s \subset \p_\t$) and 
$\e: \cell \times \cell \to \{ 0, \pm 1\}$ is an incidence function of $\cell$. Then
$$
\cpx I_R: 0 \longto I^{-d}_R \longto^\partial I^{-d + 1}_R \longto^\partial \cdots 
\longto^\partial I^0_R \longto 0
$$
is a complex. The following is the main result of \cite{OY}. Even if $R$ itself 
is an affine semigroup ring, this does not hold in the non-normal case.   

\begin{thm}[{\cite[Theorem~5.2]{OY}}]\label{sec:ishida}
If $R$ is cone-wise normal, then 
$\cpx I_R$ is quasi-isomorphic to the normalized dualizing complex $\cpx D_R$ of $R$.
\end{thm}

While the word ``dualizing complex" sometimes means its isomorphism class in the derived category,  
we use the convention that a dualizing complex 
$\cpx D_A$ of a noetherian ring $A$ is a complex of  injective $A$-modules.

\medskip

For $\s \in \cell$, set $T_\s := \set{t^a}{a \in \M_\s} \subset R$. 
Then $T_\s$ forms a multiplicatively closed subset consisting of monomials. 
Well, set 
$$
L^i_R := \Dsum_{\substack{\s \in \cell \\ \dim \s = i-1}}T_\s^{-1}R
$$
and  define $\partial:L^i_R \to L^{i+1}_R$
by
$$
\partial(x) = \sum_{\substack{\t \ge \s \\ \dim \t = i}} \e(\t,\s) \cdot g_{\t, \s}(x)
$$
for $x \in T^{-1}_\s R \subset L^i_R$,
where  
$g_{\t, \s}$ is a natural map $T^{-1}_\s R \to T^{-1}_\t R$
for $\s \le \t$.
Then $(\cpx L_R, \partial)$ forms a complex in $\Lgr R$:
$$
\cpx L_R: 0 \longto L^0_R \longto^\partial L^1_R \longto^\partial \cdots \longto^\partial 
L^d_R \longto 0. 
$$ 

We set $\m := (t^a \mid 0 \not= a \in \sM)$. This is a maximal ideal of $R$.

\begin{prop}[{cf. \cite[Theorem~4.2]{IR}}]
For $M \in \Mod R$, we have $H^i_\m(M) \cong H^i(M \tns_R \cpx L_R)$ for all $i$.
In other words, $\cpx L_R$ is a C\v ech complex of $R$ with respect $\m$.  
\end{prop}

The proof  for the $\ZZ^n$-graded case given in \cite{IR} also 
works here. 
See \cite[Proposition~3.2]{OY} for detail.

Note that the localization $T_{\s}^{-1}R$ is a $\zM$-graded $R$-module, and 
$\cpx L_R$ is a  $\zM$-graded complex. In fact, if we set 
$$\MM - \M_\s := \{  \, a-b \mid a \in \M_\t, \, b \in \M_\s, \, \t \geq \s  \, \} \subset \szM$$
Then $$T_{\s}^{-1}R = \bigoplus_{c \in \MM - \M_\s} \kk \, t^c.$$

We can define the \term{($\zM$-graded) Matlis duality functor} 
$(-)^\vee: \Lgr R \to (\Lgr R)^\op$ as follows; 
Let $M \in \Lgr R$.  For $a \in \szM$ and $b \in \sM$ such that $a+b$ exists,  
$(M^\vee)_a$ is the $\kk$-dual space of $M_{-a}$,  
and the multiplication map $(M^\vee)_a \ni x \mapsto t^b x \in (M^\vee)_{a+b}$ is the $\kk$-dual 
of $M_{-a-b} \ni y \mapsto t^b y \in M_{-a}$.  

In \cite[Proposition~5.5]{OY}, we actually showed the following. 

\begin{prop}\label{normal case}
If $R$ is cone-wise normal, then the Matlis dual $(\L_R)^\vee$ of $\cpx L_R$
is quasi-isomorphic to the normalized dualizing complex of $R$. 
\end{prop}

\begin{proof}
Under the notation of \cite{OY}, we have $\L_R \cong {\mathbf R} \Gamma_\m R$ and 
$\cpx D_R = \DD(R)$. In the proof of \cite[Proposition~5.5]{OY}, it is shown 
that $\cpx I_R$ is a ($\ZZ \MM$-graded) subcomplex of $(\L_R)^\vee$, 
and they are quasi-isomorphic. Hence $(\cpx L_R)^\vee$ is quasi-isomorphic 
to $\cpx D_R$ by Theorem~\ref{sec:ishida}.
\end{proof}

Since $R$ is not a graded ring in the usual sense, the above result is not trivial. 
The purpose of this paper is to show that it also holds in the non-normal case. 

\section{Main Theorem and Proof}
Let the notation be as in the previous section, in particular, $R:= \kk[\MM]$ 
is the toric face ring of Krull dimension $d$.    

To describe the Matlis dual of the localization $T_\s^{-1} R$ 
for $\s \in \cell$ explicitly, set  
$$\M_\s - \MM:= \{  \, a -b \mid a \in \M_\s, \, b \in \M_\t, \, \t \geq \s  \, \} \subset \szM$$
For $c \in \M_\s -\MM$, let $\ts^c$ be a basis element with degree $c \in \szM$, and 
$$E_\s(\MM):= \bigoplus_{ c \in \M_\s - \MM } \kk \, \ts^c.$$
Then we can regard $E_\s(\MM)$ as a $\ZZ\MM$-graded $R$-module by 
$$
t^a \cdot \ts^c=
\begin{cases}
\ts^{a+c} & \text{if $a, c \in \ZZ \M_\t$ for some $\t \geq \s$ and $a+c \in \M_\s -\MM$,}\\
0 & \text{otherwise.}
\end{cases}
$$
In fact, $E_\s(\MM)$ is the Matlis dual of the localization $T_\s^{-1}R$.
As shown in the proof of \cite[Theorem~5.1]{IR}, 
if $\MM$ comes from a fan in $\RR^n$ (i.e., $R$ has a nice $\ZZ^n$-grading), 
$E_\s(\MM)$ is the injective envelope of $\kk[\s]$ in the 
category of $\ZZ^n$-graded $R$-modules.  

For $\s, \t \in \cell$ with $\s \geq \t$, 
$$
t_\s^a \mapsto 
\begin{cases}
t_\t^a & \text{if $a \in \M_\t -\MM,$}\\
0 & \text{otherwise}
\end{cases}
$$
gives  an $R$-homomorphism $E_\s(\MM) \to E_\t(\MM)$, 
which is the Matlis dual $g_{\s, \t}^\vee$ of the natural map 
$g_{\s, \t}: T_\t^{-1}R \to T_\s^{-1} R$. 

Hence the Matlis dual $\cpx J_R := (\cpx L_R)^\vee$ of $\cpx L_R$ has the following form. 
$$\J_R : 0 \to \bigoplus_{\dim \s = d-1} E_\s(\MM) \to 
\bigoplus_{\dim \s = d-2} E_\s(\MM) \to \cdots \to 
\bigoplus_{\dim \s = 0} E_\s(\MM) \to E_\emptyset(\MM) \to 0.$$
The differentials are give by  
$$E_\s(\MM) \ni 
x \mapsto \sum_{\substack{\dim \t = \dim \s-1 \\ \t \leq \s}}\e(\s,\t) \cdot g_{\s, \t}^\vee(x)
\in \bigoplus_{\dim \t = \dim \s -1} E_\t(\MM), $$
where $\e$ is the incidence function of $\cell$.
We put the cohomological degree of $\bigoplus_{\dim \s = i-1} E_\s(\MM)$ to $-i$.
The following is a main theorem of this paper. 

\begin{thm}\label{main}
The complex $\J_R$ is quasi-isomorphic to the normalized dualizing complex $\cpx D_R$ of $R$.  
\end{thm}

\begin{rem}\label{Jfs}
When $R$ itself is an affine semigroup ring, the above theorem was given by Ishida~\cite{I} (see also 
\cite{SS}). More precisely, for the semigroup ring $R:=\kk[\M]$ of an affine semigroup $\M \subset \ZZ^n$,  
$\cpx J_R$ is a $\ZZ^n$-graded normalized dualizing complex and quasi-isomorphic to the 
usual (i.e., non-graded) one. More generally, if $\MM$ comes from a fan, then 
the theorem was given by Ichim and R\"omer (\cite[Theorem~5.1]{IR}) by standard argument using the 
graded ring structure. 
\end{rem}

For $\s, \t  \in \cell$ with $\t \geq \s$, let $E_\s(\M_\t)$ be the $\ZZ \MM$-graded 
Matlis dual of the localization $T_\s^{-1}\kk[\t]$ of $\kk[\t]=R/\p_\t$. 
Since $\kk[\t]$ is a quotient of $R$, $E_\s(\M_\t)$ is a submodule of $E_\s(\MM)$ with 
a $\kk$-basis $\{ \,  t_\s^{a-b} \mid a \in \M_\s, \, b \in \M_\t \, \}.$

By construction, we have $\p_\t E_\s(\M_\t) =0$ and 
$E_\s(\M_\t)$ can be seen as a $\ZZ^{\dim \t +1}$-graded $\kk[\t]$-module. 
In this case, $E_\s(\M_\t)$ is the injective envelope of $\kk[\M_\s]$ in the 
category of $\ZZ^{\dim \t +1}$-graded $\kk[\M_\t]$-modules. 

\begin{lem}\label{Hom(s,J)} 
For $\s, \t \in \cell$, we have 
$$\Hom_{R}(\kk[\t], E_\s(\MM)) \cong  \begin{cases}
E_\s(\M_\t) \ ( \,  \cong (T_\s^{-1} \kk[\t])^\vee \, )& \text{if $\s \leq \t$,}\\
0 & \text{otherwise.}
\end{cases}$$
\end{lem}

\begin{proof}
Assume that $\sigma \not \leq \t$. Then we can take $a \in \M_\s \setminus \M_\t$. 
Clearly, if $b \in \M_\s -\MM$, then $a+b \in \M_\s -\MM$, that is, 
$t^a \cdot t_\sigma^b \ne 0$ in $E_\s(\MM)$. Since $t^a \in \p_\t$, we have  
$\Hom_{R}(\kk[\t], E_\s(\MM)) = 0$. 

Next, assume that  $\sigma \leq \t$. The inclusion 
$$\Hom_{R}(\kk[\t], E_\s(\MM)) \cong \{ \, x \in E_\s(\MM) \mid \p_\t x = 0 \, \}
\supset E_\s(\M_\t)$$ 
is clear. 
To show the opposite inclusion, take $x \in E_\s(\MM) \setminus E_\s(\M_\t)$. 
We may assume that $x = t_\s^{a-b}$ for some $a \in \M_\s$ and $b \not \in \M_\t$. 
Then $t^b \in \p_\t$ and $t^b \cdot t_\s^{a-b} = t_\s^a \ne 0$. 
\end{proof}

The next result easily follows from the above discussion (and Remark~\ref{Jfs}). 

\begin{lem}
For $\s \in \cell$, the complex $\inHom_R(\fring{\s}, \cpx J_R)$ is isomorphic to $\cpx J_{\fring{\s}}$, 
and quasi-isomorphic to the dualizing complex 
$\cpx D_{\fring{\s}} = \inHom_R(\fring{\s}, \cpx D_R)$ of $\fring{\s}$.  
\end{lem}

For each $\s \in \cell$, set $\bM_\s:= \ZZ^{\dim \s +1} \cap C_\s$. 
Then $\bMM := \{ \, \bM_\s \, \}_{\s \in \cell}$ is a monoidal complex supported by $\cell$ again. 
Let $\oR := \kk[\bMM]$ be the toric face ring of $\bMM$. 
For the monomial prime ideal $\oP_\s$ of $\oR$ associated with $\s \in \cell$, 
we have $\oR/\oP_\s \cong \kk[\bM_\s]$ and this is the normalization $\overline{\fring{\s}}$ 
of $\fring{\s}$. So we denote $\oR/\oP_\s$ by $\overline{\fring{\s}}$. 
Since $\oR$ is cone-wise normal, $\cpx J_{\oR}$ is quasi-isomorphic to $\cpx D_{\oR}$ 
by Theorem~\ref{sec:ishida}. 
Moreover, we have the following.

\begin{lem}
There is a quasi-isomorphism $\psi: \cpx J_{\oR} \to \cpx D_{\oR}$ such that the 
induced map $\psi_\s := \inHom_{\oR}(\overline{\fring{\s}}, \psi) : \cpx J_{\overline{\fring{\s}}} \to 
\cpx D_{\overline{\fring{\s}}}$ is a quasi-isomorphism for all $\s \in \cell$.  
\end{lem}

\begin{proof}
In \cite{OY}, we showed that $\cpx I_{\oR}$ can be seen as a subcomplex of $\cpx D_{\oR}$. 
This gives a quasi-isomorphism $\eta: \cpx I_{\oR} \to \cpx D_{\oR}$ 
such that the induced map $\eta_\s := \inHom_{\oR}(\ofs, \eta) : \cpx I_{\ofs} \to \cpx D_{\ofs}$ is 
a quasi-isomorphism again for all $\s \in \cell$.  

Since $\oR$ is cone-wise normal, $\cpx I_{\oR}$ is a $\zM$-graded subcomplex of $\cpx J_{\oR}$, 
and the chain map $\iota:  \cpx I_{\oR} \hookrightarrow \cpx J_{\oR}$ 
is a quasi-isomorphism as pointed out in the proof of Proposition~\ref{normal case}. 
The diagram $\cpx J_{\oR} \stackrel{\iota}{\longleftarrow} \cpx I_{\oR} \stackrel{\eta}{\longrightarrow}
\cpx D_{\oR}$ gives an isomorphism $\cpx J_{\oR} \to \cpx D_{\oR}$ in $\Db(\Mod R)$.  
Since $\cpx D_{\oR}$ is a complex of injective $R$-modules, there is an actual chain map 
$\psi : \cpx J_{\oR} \to \cpx D_{\oR}$ giving this isomorphism. 
Clearly, $\psi$ is a quasi-isomorphism and $\eta$ is homotopic to $\psi \circ \iota$. 
It is easy to see that $\iota_\s := \inHom_{\oR}(\ofs, \iota): \cpx I_{\ofs} \to \cpx J_{\ofs}$ 
is a quasi-isomorphism, and $\eta_\s$ is homotopic to $\psi_\s \circ \iota_\s$. 
Since $\eta_\s$ and $\iota_\s$ are quasi-isomorphisms, so is $\psi_\s$. 
\end{proof}

Since $\oR$ is finitely generated as an $R$-module, we have $\DC_{\oR} = \inHom_R(\oR, \DC_R)$. 
Via the canonical injection $R \hookrightarrow \oR$, we have a chain map 
$\lambda :\DC_{\oR} = \inHom_R(\oR, \DC_R) \longrightarrow \inHom_R(R,\DC_R)=\DC_R$. 
Similarly, for each $\s \in \cell$, the injection $\fs \hookrightarrow \ofs$ 
induces a chain map $\lambda_\s : \cpx D_{\ofs} \to \cpx D_{\fs}$. 

As a $\ZZ^{\dim \s +1}$-graded version of the well-known fact  
$\cpx D_{\ofs} = \inHom_{\fs}(\ofs, \cpx D_{\fs})$, we have 
$\cpx J_{\ofs} = \inHom_{\fs}(\ofs, \cpx J_{\fs}).$ 
Similarly, we have a chain map $\mu_\s: \cpx J_{\ofs} \to \cpx J_{\fs}$ 
which is the $\ZZ^{\dim \s +1}$-graded version of $\lambda_\s$. 

\begin{lem}\label{commutative diagram}
For the quasi-isomorphism $\psi_\s : \cpx J_{\ofs} \to \cpx D_{\ofs}$, 
we have a quasi-isomorphism $\phi_\s : \cpx J_{\fs} \to \cpx D_{\fs}$ 
which makes the following diagram commutative.
$$
\xymatrix{
J^\bullet_{\ofs} \ar[r]^{\psi_\sigma} \ar[d]_{\mu_\sigma} 
& D^\bullet_{\ofs} \ar[d]^{\lambda_\sigma} \\
J^\bullet_{\fs} \ar[r]_{\phi_\sigma} & D^\bullet_{\fs}
}
$$ 
\end{lem}

\begin{proof}
Let $\xi : \cpx J_{\fs} \to \cpx D_{\fs}$ be an arbitrary quasi-isomorphism. 
Since $\ofs$ is a $\ZZ^{\dim \s +1}$-graded $\fs$-module and $\cpx J_{\fs}$ 
is a complex of $\ZZ^{\dim \s +1}$-graded injective $\fs$-modules, 
$\xi$ gives a quasi-isomorphism 
$$\xi_*: \cpx J_{\ofs} = \inHom_{\fs}(\ofs, \cpx J_{\fs}) \longrightarrow 
\inHom_{\fs}(\ofs, \cpx D_{\fs}) = \cpx D_{\ofs}.$$
Since, as is well-known, 
$\Hom_{\Db(\Mod \ofs)}(\cpx D_{\ofs}, 
\cpx D_{\ofs}) = \ofs$, we have $\psi_\s = c \, \xi_*$ for some $0 \ne c \in \kk$. 
Hence $\phi_\s := c \, \xi$ satisfies the expected condition. 
\end{proof}

For each $i \in \ZZ$,  $J^i_R$ is a $\zM$-graded submodule of 
$J^i_{\oR}$ (here we regard $J^i_{\oR}$ as an $R$-module), moreover, 
$J^i_R$ is a direct summand of $J_{\oR}^i$.  
However  $\cpx J_R$ is NOT a subcomplex 
of $\cpx J_{\oR}$. Let $\kappa: \cpx J_R \dashrightarrow \cpx J_{\oR}$ be 
the component-wise injection  (since this is not a chain map, we use the 
symbol ``$\dashrightarrow$"). For the similar map $\kappa_\s: \cpx J_{\fs} 
\dashrightarrow \cpx J_{\ofs}$, we have $\mu_\s^i \circ \kappa^i_\s = \operatorname{Id}$ for all $i$. 

\begin{lem}\label{composition}
The composition of  
$$\cpx J_R \stackrel{\kappa}{\dashrightarrow} \cpx J_{\oR} \stackrel{\psi}{\longrightarrow} \cpx D_{\oR} 
\stackrel{\lambda}{\longrightarrow} \cpx D_R$$ 
is a chain map. 
\end{lem}

\begin{proof}
For any  $x \in J^i_R$, there is some $\s \in \cell$ such that $\p_\s x =0$. 
Regarding $\cpx J_{\fs} = \inHom_R(\fs, \cpx J_R)$ as a subcomplex of $\cpx J_R$, 
we have $x \in J^i_{\fs}$. 
Since $\cpx J_{\ofs}$ (resp.  $\cpx D_{\ofs}$ and $\cpx D_{\fs}$) can be seen 
as a subcomplex of $J_{\oR}^i$ (resp. $\cpx D_{\oR}$ and $\cpx D_R$), 
we have the following commutative diagram.
$$
\xymatrix{
J^i_{\fs} \ar[r]^{\kappa_\sigma} \ar@{^(->}[d] & J^i_{\ofs} \ar[r]^{\psi_\sigma} \ar@{^(->}[d] &
D^i_{\ofs} \ar[r]^{\lambda_\sigma} \ar@{^(->}[d] & D^i_{\fs}  \ar@{^(->}[d]\\
J^i_R \ar[r]_{\kappa} &J^i_{\oR} \ar[r]_{\psi} & D^i_{\oR} \ar[r]_{\lambda} & D^i_R
}
$$
By Lemma~\ref{commutative diagram}, we have $\lambda^i_\s \circ \psi^i_\s \circ \kappa^i_\s 
= \phi_\s^i \circ \mu^i_\s \circ \kappa^i_\s = \phi_\s^i$. 
Since $\phi_\s$ is a chain map, we are done. 
\end{proof}

We denote the chain map $\cpx J_R \to \cpx D_R$ constructed in Lemma~\ref{composition} by $\phi$. 
To prove Theorem~\ref{main}, we will show that $\phi$
is a quasi-isomorphism by a slightly indirect way.
 
\medskip

For each $a \in \sM$, there is a unique minimal element among the cells 
$\s \in \cell$ such that $a \in \M_\s$. We denote this minimal cell by $\supp(a)$.   

\begin{dfn}
An $R$-module $M \in \Lgr R$ is said to be \term{squarefree} if it is $\MM$-graded 
(i.e., $M=\bigoplus_{a \in \sM} M_a$), finitely generated,  and the multiplication map $M_a \ni  x \mapsto t^b x \in M_{a+b}$  
is bijective for all $a,b\in \sM$ such that $a+b$ exists and 
 $\supp(a+b) = \supp(a)$.
\end{dfn}

The notion of squarefree modules over a normal semigroup ring was introduced by the author (\cite{Y2}), 
and many applications have been found.  In \cite{OY}, squarefree modules over 
a cone-wise normal toric face ring play a key role. 
Contrary to the (cone-wise) normal case, the derived category of squarefree modules is {\it not} closed 
under the duality ${\mathbf R} \Hom_R(-, \DC_R) = \inHom_R(-, \DC_R)$.  
However, these modules still enjoy some nice properties.

\begin{lem}[{cf. \cite[Lemma~4.2]{OY}}]\label{sec:maps}
Let $M \in \Sq R$. 
Then for $a,b \in \sM$ with $\supp(a) \ge \supp(b)$, there exists a $\kk$-linear map
$\varphi^M_{a,b}:M_b \to M_a$ satisfying the following properties:
\begin{enumerate}
\item If $a= b+c$, then $\varphi^M_{a,b}$ coincides with the multiplication map 
$M_b \ni x \mapsto t^c \, x \in M_a$; 
\item $\varphi^M_{a,b}$ is bijective if $\supp(a) = \supp(b)$;
\item $\varphi^M_{a,b}\circ \varphi^M_{b,c} = \varphi^M_{a,c}$
for $a,b,c \in \sM$ with $\supp(c) \le \supp(b) \le \supp(a)$. 
\end{enumerate}
\end{lem}

\begin{proof}
The proof for the normal case works here, but we repeat it for the reader's 
convenience. Set $\varphi^M_{a+b, b}:M_b \to M_{a+b}$ to be the multiplication map 
$M_b \ni x \mapsto t^a  x \in M_{a+b}$, and   
define $\varphi^M_{a+b,a}: M_a \to M_{a+b}$ in the same way. 
Since $\supp(a)=\supp(a+b)$, $\varphi^M_{a+b,a}$ is bijective and we can put 
$\varphi^M_{a,b}:= (\varphi^M_{a+b, a})^{-1} \circ \varphi^M_{a+b,b}$. 
\end{proof}

Let $\Sq R$ be the full subcategory of $\Lgr R$ consisting of squarefree modules. 
By virtue of the above lemma, \cite[Lemma~4.2]{OY} remains true in the present case.  

\begin{lem}[{c.f. \cite[Lemma~4.2]{OY}}]
The category $\Sq R$ is equivalent to the category of finitely generated left $\Lambda$-modules,  
where $\Lambda$ is the incidence algebra of the poset $\cell$ over $\kk$. 
Hence $\Sq R$ is an abelian category with enough injectives, 
and indecomposable injectives are objects isomorphic to $\fs$ for some $\s \in \cell$.  
The injective dimension of any object is at most $d$.
\end{lem}

Recall that if $M$ is a $\zM$-graded $R$-module, then the localization $T_\s^{-1}M$ is also.  
Since $\cpx L_R$ is a complex of flat $R$-modules,  $(- \otimes_R \cpx L_R)$ gives an exact functor 
$\Db(\Lgr R) \to \Db(\Lgr R)$.  Composing this one and the Matlis duality, we have an exact  
functor  $(- \otimes_R \cpx L_R)^\vee: \Db(\Lgr R) \to \Db(\Lgr R)^\op$. 

Let $\InjSq$ be the full subcategory of $\Sq R$ consisting of all injective objects, that is, 
finite direct sums of $\kk[\s]$ for various $\s \in \cell$. 
As is well-known (cf. \cite[Proposition~I.4.7]{Ha}), the bounded homotopy category $\Cb(\InjSq)$ 
is equivalent to $\Db(\Sq R)$. It is easy to see that the functor 
$(- \otimes_R \cpx L_R)^\vee: \Db(\Sq R) \to \Db(\Lgr R)^\op$ can be identified with  
$\inHom_R(-, \cpx J_R): \Cb(\InjSq) \to \Db(\Lgr R)^\op$ by Lemma~\ref{Hom(s,J)}. 
Via the forgetful functor $\Lgr R \to \Mod R$, we get an exact functor 
$$\inHom_R(-, \cpx J_R): \Cb(\InjSq) \to \Db(\Mod R)^\op.$$ 

Since $\cpx D_R$ is a complex of injective $R$-modules, $\inHom_R(-, \cpx D_R)$ 
gives an exact functor $\Db(\Mod R) \to \Db(\Mod R)^\op$. Similarly, 
we have an  exact functor 
$$\inHom_R(-, \cpx D_R): \Cb(\InjSq) \to \Db(\Mod R)^\op.$$
The chain map $\phi: \cpx J_R \to \cpx D_R$ gives a natural transformation 
$$\Phi: \inHom_R(-, \cpx J_R) \to \inHom_R(-, \cpx D_R).$$

\begin{thm}\label{Phi}
The natural transformation $\Phi$ is an natural isomorphism. 
\end{thm}

\begin{proof}
By virtue of \cite[Proposition 7.1]{Ha}, it suffices to show that 
$\Phi(\kk[\s]): \cpx J_{\fring{\s}} = \inHom_R(\fring{\s}, \cpx J_R) \to \inHom_R(\fring{\s}, \cpx D_R) = \cpx D_{\fring{\s}}$  
is quasi-isomorphism for all $\s \in \cell$.    Since $\Phi(\kk[\s]) = \inHom_R(\fring{\s}, \phi)$, it is factored as 
$$\cpx J_{\fring{\s}} \stackrel{\kappa_\s}{\dashrightarrow} \cpx J_{\overline{\fring{\s}}} 
\stackrel{\psi_\sigma}{\longrightarrow} \cpx D_{\overline{\fring{\s}}} 
\stackrel{\lambda_\s}{\longrightarrow} \cpx D_{\fring{\s}},$$
while $\kappa_\s$ is just a ``component-wise map".  As shown in the proof of Lemma~\ref{composition}, 
this coincides with the quasi-isomorphism $\phi_\s$ of Lemma~\ref{commutative diagram}.  
\end{proof}

\noindent{\it The proof of Theorem~\ref{main}.}
The theorem follows from Theorem~\ref{Phi}. In fact, $R \in \Sq R$ and 
$\phi: \cpx J_R \to \cpx D_R$ coincides with 
$\Phi(R): \inHom_R(R, \cpx J_R) \to \inHom_R(R, \cpx D_R)$. 
\qed

\begin{cor}
$R$ is Cohen-Macaulay  if and only if so is the local ring $R_\m$. 
\end{cor}

\begin{proof} By Theorem~\ref{main}, 
$R$ is Cohen-Macaulay if and only if $H^i(\cpx J_R) = 0$ for all $i \ne -d$. 
Since $H^i(\cpx J_R)$ is $\zM$-graded, $H^i(\cpx J_R) \ne 0$ implies 
$\m \in \operatorname{Supp} (H^i (\cpx J_R))$ by Lemma~\ref{associated prime} below. 
Hence $R$ is Cohen-Macaulay, if and only if $H^i(\cpx J_R) \otimes R_\m = 0$ for all $i \ne -d$, 
if and only if $R_\m$ is Cohen-Macaulay. 
\end{proof}

\begin{lem}\label{associated prime}
If $M \in  \Lgr R$ is finitely generated, 
then any associated prime of $M$ is of the form $\p_\s$ for some $\s \in \cell$.
\end{lem}

\begin{proof}
Let $\p$ be an associated prime of $M$.  
Since any minimal prime of $R$ is of the form $\p_\t$ for a maximal cell $\t \in \cell$  
(see \cite{OY}), there is some $\t \in \cell$ with $\p_\t \subset \p$. 
The submodule $M':=\{ \, y \in M \mid \p_\t y =0 \, \}$ of $M$ is a $\ZZ^{\dim \t +1}$-graded 
$\kk[\t]$-module, and the image $\bar{\p}$ of $\p$ in $\kk[\t]$ 
is an associated prime of $M'$. Hence $\p$ is  $\ZZ^{\dim \t +1}$-graded, and $\p=\p_\s$ 
for some $\s \in \cell$ with $\s \leq \t$. 
\end{proof}

%
%
%

%
%
%
%
\end{document}